\documentclass{amsart}
\usepackage{float}

\newtheorem{theorem}{Theorem}[section]

\newtheorem{corollary}[theorem]{Corollary}

\theoremstyle{definition}
\newtheorem{definition}[theorem]{Definition}

\theoremstyle{remark}
\newtheorem{remark}[theorem]{Remark}

\numberwithin{equation}{section}

\newcommand{\F}{\mathbb{F}}
\newcommand{\Z}{\mathbb{Z}}

\newcommand{\R}{\Z_4+u\Z_4}

\begin{document}

\title[Linear codes over $\Z_4+u\Z_4$]{Linear Codes over $\Z_4+u\Z_4$: MacWilliams identities, projections, and formally self-dual codes}

\author{Bahattin Yildiz}
\author{Suat Karadeniz}
\address{Department of Mathematics, Fatih University, 34500
Istanbul, Turkey}
\email{byildiz@fatih.edu.tr, skaradeniz@fatih.edu.tr}

\subjclass[2000]{Primary 94B05; Secondary 11T71}

\keywords{complete
weight enumerator, MacWilliams identities, projections, lifts, formally self-dual codes, codes over rings}

\footnote{This work has been partially presented in the proceedings of the 13th International Workshop on Algebraic and combinatorial coding theory, Pomorie, Bulgaria, 2012.}

\begin{abstract}
Linear codes are considered over the ring $\Z_4+u\Z_4$, a non-chain extension of $\Z_4$. Lee weights, Gray maps for these codes are defined and MacWilliams identities for the complete, symmetrized and Lee weight enumerators are proved. Two projections from $\R$ to the rings $\Z_4$ and $\F_2+u\F_2$ are considered and self-dual codes over $\R$ are studied in connection with these projections. Finally three constructions are given for formally self-dual codes over $\R$ and their $\Z_4$-images together with some good examples of formally self-dual $\Z_4$-codes obtained through these constructions.
\end{abstract}

\maketitle



\section{Introduction}

Codes over rings have been a common research topic in coding theory. Especially after the appearance of \cite{Sloane:Gray}, a lot of research went towards studying codes over $\Z_4$. The rich algebraic structure of rings has allowed researchers get good results in coding theory. The rings studied have varied over the recent years, but codes over $\Z_4$ remain a special topic of interest in coding theory because of their relation to lattices, designs, cryptography and their many applications. There is a vast literature on codes over $\Z_4$; for some of the works done in this direction we refer to \cite{Duursma}, \cite{Gulliver}, \cite{Huffman}, \cite{Wan}, \cite{Wolf}, etc.

Recently, several families of rings have been introduced in coding theory, rings that are not finite chain but are Frobenius. These rings have a rich algebraic structure and they lead to binary codes with large automorphism groups and in some cases new binary self-dual codes(\cite{YKilk}, \cite{us}, \cite{68}, \cite{bordered}).

$\F_2+u\F_2$ is a size 4 ring that also has generated a lot of interest among coding theorists starting with \cite{Dougherty}. There is an interesting connection between $\Z_4$ and $\F_2+u\F_2$. Both are commutative rings of size $4$, they are both finite-chain rings and they have both been studied quite extensively in relation to coding theory. Some of the main differences between these two rings are that their characteristic is not the same, $\F_2$ is a subring of $\F_2+u\F_2$ but not that of $\Z_4$ and the Gray images of $\Z_4$-codes are usually not linear while the Gray images of $\F_2+u\F_2$-codes are linear.

Inspired by this similarity(and difference) between the two rings, and our previous works on the non-chain size 16 ring $\F_2+u\F_2+v\F_2+uv\F_2$,
we study codes over the ring $\Z_4+u\Z_4$ in this work. It turns out that this ring is also a non-chain, commutative ring of size $16$ but with characteristic 4. The ideal structure turns out to be very similar to that of $ \F_2+u\F_2+v\F_2+uv\F_2$.

We define linear Gray maps from $\R$ to $\Z_4^2$, extend the Lee weight from $\Z_4$ to $\R$ and we define and prove the MacWilliams identities  for all the weight enumerators involved.

We also study self-dual codes over $\R$ and the images of these codes under several maps. In particular, we introduce two projections, one from $\R$ to $\Z_4$ and the other from $\R$ to $\F_2+u\F_2$. We study certain properties of the projections in terms of the minimum weights and self-duality.

The last part of the paper is about three constructions for formally self-dual codes over $\R$ whose Gray images are also formally self-dual over $\Z_4$. We tabulate some good formally self-dual codes over $\Z_4$ obtained from formally self-dual codes over $\R$ through two of the constructions.

\section{Linear Codes over $\Z_4+u\Z_4$}
\subsection{The Ring $\Z_4+u\Z_4$}
$\Z_4+u\Z_4$ is constructed as a commutative, characteristic $4$ ring with $u^2 = 0$ and it is clear that
$\R \cong \Z_4[x]/(x^2)$.

Its units are given by
$$\{1,1+u,1+2u,1+3u,3,3+u,3+2u,3+3u\},$$

while the non-units are
$$\{0,2,u,2u,3u,2+u,2+2u,2+3u\}.$$

It has a total of $6$ ideals given by
\begin{equation}
I_0=\{0\} \subseteq I_{2u}=2u (\R)=\{0,2u\} \subseteq
I_u,I_2,I_{2+u} \subseteq I_{2,u} \subseteq I_1=\R
\end{equation}
where
\begin{align*}
I_u & = u(\R)=\{0,u,2u,3u\}, \\
I_2 & =2(\R)=\{0,2,2u,2+2u\}, \\
I_{2+u} & =(2+u)(\R)=\{0,2+u,2u,2+3u\} \\
I_{2,u} &= \{0,2,u,2u,3u,2+u,2+2u,2+3u\}.
\end{align*}

It is clear that $\R$ is a local ring with  $I_{2,u}$ as its maximal ideal.  The residue field is given by $(\R)/I_{2,u} \simeq \F_2$. Since $Ann(I_{2,u}) = \{0,2u\}$, and this has dimension $1$ over the residue field, we have from \cite{Wood}
that
\begin{theorem}
$\R$ is a local Frobenius ring.
\end{theorem}
The ideal $\langle 2,u \rangle$ is not principal and the ideals $\langle 2 \rangle$ and $\langle u \rangle $ are not related via inclusion, which means that $\R$ is not a finite chain ring or a principal ideal ring.

We divide the units of $\R$ into two subsets $\mathfrak{U}_1$ and $\mathfrak{U}_2$ calling them units of first type and second type, respectively, as follows:
\begin{equation}
 \mathfrak{U}_1 = \{1,3,1+2u,3+2u\}
\end{equation}
and
\begin{equation}
 \mathfrak{U}_2 = \{1+u,3+u,1+3u,3+3u\}.
\end{equation}
The reason that we distinguish between the units is the following observation that can easily be verified:
\begin{equation}\label{squareelements}
\forall a\in \R, \:\:\:\: a^2 = \left \{
\begin{array}{lll}
0 & \textrm{if $a$ is a non-unit}
\\
1 & \textrm {if $a \in \mathfrak{U}_1$}
\\
1+2u & \textrm {if $a \in \mathfrak{U}_2$.}
\end{array}\right.
\end{equation}

\subsection{Linear Codes over $\Z_4+u\Z_4$, the Lee weight and the Gray map}
\begin{definition}
A linear code $C$ of length $n$ over the ring $\R$ is a
$\R$-submodule of $(\R)^n.$
\end{definition}
Since $\Z_4+u\Z_4$ is not a finite chain ring, we cannot define a standard generating matrix for linear codes over $\R$.

In the case of $\F_2+u\F_2$, the Lee weight was defined in \cite{Dougherty} as
$w(0)=0, w(1)=w(1+u) = 1, w(u)= 2$, and accordingly a
Gray map from $(\F_2+u\F_2)^n$ to $\F_2^{2n}$ was defined by sending
$\overline{a}+u \overline{b}$ to
$(\overline{b},\overline{a}+\overline{b})$ with $\overline{a},\overline{b} \in \F_2^n$.  We will adopt a
similar technique here. We will generalize the Gray map and define the
weight in such a way that will give us a distance preserving
isometry. Define $\phi:(\R)^n \rightarrow \Z_4^{2n}$ by
\begin{equation}
\phi(\overline{a}+u \overline{b}) = (\overline{b},\overline{a}+\overline{b}), \:\:\:\:\: \overline{a}, \overline{b} \in \Z_4^n.
\end{equation}

We now define the Lee weight $w_L$ on $\R$ by letting
$$w_L(a+ub) = w_L((b,a+b)),$$
where $w_L((b,a+b))$ describes the usual Lee weight on $\Z_4^2$. The Lee distance is defined accordingly. Note that with this definition of the Lee weight and the Gray map we have the following main theorem:
\begin{theorem}\label{GrayImage}
$\phi:(\R)^n \rightarrow \Z_4^{2n}$ is a distance preserving linear isometry. Thus, if $C$ is a linear code over $\R$ of length $n$, then $\phi(C)$ is a linear code over $\Z_4$ of length $2n$ and the two codes have the same Lee weight enumerators.
\end{theorem}
Throughout the paper we will use the notation $w_L$ to denote the Lee weight on $\R$, as well as $\Z_4$ and $\F_2+u\F_2$.
\section{The Dual, The Complete Weight Enumerator and MacWilliams Identities}
\subsection{The Dual of linear codes over $\R$}
First take the Euclidean inner product on $(\R)^n$ by taking
\begin{equation}
\langle(x_1, x_2, \dots, x_n),(y_1,y_2, \dots, y_n)\rangle = x_1y_1+x_2y_2+
\dots + x_ny_n
\end{equation}
where the operations are performed in the ring $\R$.

We are now ready to define the dual of a linear code $C$ over
$\R$:
\begin{definition}
Let $C$ be a linear code over $\R$ of length $n$, then we define
the {\it dual} of $C$ as
$$C^{\perp}:= \large\{ \overline{y} \in (\R)^n \large | \langle\overline{y},\overline{x}\rangle = 0, \:\:\:\: \forall
\overline{x} \in C \large \}.$$
\end{definition}

Note that from the definition of the inner product, it is obvious
that $C^{\perp}$ is also a linear code over $\R$ of length $n$. Since $\R$ is a frobenius ring we also have $|C|\cdot|C^{\perp}| = 16^n$.

We next study the MacWilliams identities for codes over $\R$:

\subsection{The Complete Weight Enumerator and MacWilliams Identities}
Let $\R = \{ g_1,g_2, \dots, g_{16} \}$ be given as
$$\R=\{0,u,2u,3u,1,1+u,1+2u,1+3u,2,2+u,2+2u,2+3u,3,3+u,3+2u,3+3u\}.$$
\begin{definition}
The complete weight enumerator of a linear code $C$ over $\R$ is defined as
$$cwe_C(X_1, X_2, \dots,X_{16}) = \sum_{\overline{c} \in C} \large ( X_1^{n_{g_1}(\overline{c})}X_2^{n_{g_2}(\overline{c})} \dots
X_{16}^{n_{g_{16}}(\overline{c})}\large ) $$
\end{definition}
where $n_{g_i}(\overline{c})$ is the number of appearances of
$g_i$ in the vector $\overline{c}$.
\begin{remark}
Note that $cwe_C(X_1, X_2, \dots,X_{16})$ is a homogeneous
polynomial in $16$ variables with the total degree of each monomial
being $n$, the length of the code. Since $\overline{0} \in C$, we
see that the term $X_1^{n}$ always appears in $cwe_C(X_1, X_2,
\dots,X_{16})$.
\end{remark}

We also observe that
\begin{equation}
cwe_C(1, 1, \dots,1) = |C|,
\end{equation}
and
\begin{equation}
cwe_C(a, 0, \dots,0) = a^n.
\end{equation}

The complete weight enumerator gives us a lot of information about
the code.

Now, since $\R$ is a Frobenius ring, the MacWilliams identities for the complete weight enumerator hold. To find the exact identities we define the following character on $\Z_4+u\Z_4:$
\begin{definition}
Define $\chi: \R \rightarrow \mathbb{C}^{\times}$ by
$$\chi(a+bu) = i^{a+b}.$$
\end{definition}
It is easy to verify that $\phi$ is a non-trivial character when restricted to each non-zero ideal, hence it is a generating character for $\R$.

Then we make up the $16 \times 16$ matrix $T$, by letting $T(i,j) = \chi(g_ig_j).$ The matrix $T$ is given as follows:
$$T=\left [
\begin{array}{cccccccccccccccc}
1& 1& 1& 1& 1& 1& 1& 1& 1& 1& 1& 1& 1& 1& 1& 1 \\
1& 1& 1& 1& i& i& i& i& -1& -1& -1& -1& -i& -i& -i& -i\\
1& 1& 1& 1& -1& -1& -1& -1& 1& 1& 1& 1& -1& -1& -1& -1\\
1& 1& 1& 1& -i& -i& -i& -i& -1&-1& -1& -1& i& i& i& i\\
1& i& -1& -i& i& -1& -i& 1& -1& -i& 1& i& -i& 1& i& -1\\
1&i& -1& -i& -1& -i& 1& i& 1& i& -1& -i& -1& -i& 1&i\\
1& i& -1& -i& -i& 1& i& -1& -1& -i& 1& i& i& -1& i& 1\\
1& i& -1& -i& 1& i& -1& -i& 1& i& -1& -i& 1& i& -1& i \\
1& -1& 1& -1& -1& 1& -1& 1& 1& -1& 1& -1& -1& 1& -1& 1 \\
1& -1& 1& -1& -i& i& -i& i& -1& 1& -1& 1& i& -i& i& -i \\
1& -1& 1& -1& 1& -1& 1& -1& 1& -1& 1& -1& 1& -1& 1& -1\\
1& -1& 1& -1& i& -i& i& -i& -1& 1& -1& 1& -i& i& -i& i\\
1&-i& -1& i& -i& -1& i& 1& -1& i& 1& -i& i& 1& -i& -1\\
1& -i& -1& i& 1& -i& -1& i& 1& -i& -1& i& 1& -i& -1& i\\
1& -i& -1& i& i& 1& i& -1& -1& i& 1& -i& -i& -1& i& 1 \\
1& -i& -1& i& -1& i& 1& -i& 1& -i& -1& i& -1& i& 1& -i
\end{array}
\right].$$

The following theorem then follows from \cite{Wood} quite easily:
\begin{theorem}\label{cweMacWilliams}
Let $C$ be a linear code over $\R$ of length $n$ and suppose $C^{\perp}$ is its dual. Then we have
$$cwe_{C^{\perp}}(X_1,X_2, \dots, X_{16}) = \frac{1}{|C|}cwe_C(T\cdot (X_1,X_2, \dots, X_{16})^t),$$
where $()^t$ denotes the transpose.
\end{theorem}

\subsection{The Symmetrized Weight enumerator and The Lee Weight enumerator}
Since in $\Z_4$, $w_L(1) = w_L(3) = 1$, the symmetrized weight enumerator for codes over $\Z_4$ was defined as
$$swe_C(X,Y,Z) = cwe_C(X,Y,Z,Y).$$
Adopting the same idea, we will define the symmetrized weight enumerator of codes over $\R$. To do this we need the following table which gives us the elements of $\R$, their Lee weights and the corresponding variables:

\bigskip
\noindent

\begin{center}
Table 1: The Lee Weights of the elements of $\R$.

\medskip

\begin{tabular}
[l]{|c|c|c|}\hline $a$  &
 Lee Weight of $a$ &  The corresponding variable
\\ \hline $0$ & $0$ & $X_1$
\\ \hline $u$ & $2$ & $X_2$
\\ \hline $2u$ & $4$ & $X_3$
\\ \hline $3u$ & $2$ & $X_4$
\\ \hline $1$ & $1$ & $X_5$
\\ \hline $1+u$ & $3$ & $X_6$
\\ \hline $1+2u$ & $3$ & $X_7$
\\ \hline $1+3u$ & $1$ & $X_8$
\\ \hline $2$ & $2$ & $X_9$
\\ \hline $2+u$ & $2$ & $X_{10}$
\\ \hline $2+2u$ & $2$ & $X_{11}$
\\ \hline $2+3u$ & $2$ & $X_{12}$
\\ \hline $3$ & $1$ & $X_{13}$
\\ \hline $3+u$ & $1$ & $X_{14}$
\\ \hline $3+2u$ & $3$ & $X_{15}$
\\ \hline $3+3u$ & $3$ & $X_{16}$
\\ \hline
\end{tabular}
\end{center}

\medskip
So, looking at the elements that have the same weights we can define the symmetrized weight enumerator as follows:
\begin{definition}
Let $C$ be a linear code over $\R$ of length $n$. Then define the symmetrized weight enumerator of $C$ as
\begin{equation}
swe_C(X,Y,Z,W,S) = cwe_C(X,S,Y,S,W,Z,Z,W,S,S,S,S,W,W,Z,Z).
\end{equation}
\end{definition}

Here $X$ represents the elements that have weight $0$ (the 0 element); $Y$ represents the elements with weight $4$ (the element $2u$);
$Z$ represents the elements of weight $3$ (the elements $1+u$, $1+2u$, $3+2u$ and $3+3u$); $W$ represents the elements of weight $1$ (the elements $1$, $1+3u$, $3$ and $3+u$) and finally $S$ represents the elements of weight $2$ (the elements $2$, $u$, $3u$, $2+u$, $2+2u$ and $2+3u$).

Now, combining Theorem \ref{cweMacWilliams} and the definition of the symmetrized weight enumerator, we obtain the following theorem:
\begin{theorem}\label{sweMacWilliams}
Let $C$ be a linear code over $\R$ of length $n$ and let $C^{\perp}$ be its dual. Then we have
$$swe_{C^{\perp}}(X,Y,Z,W,S) = $$
$$\frac{1}{|C|}swe_C(6S+4W+X+Y+4Z, 6S-4W+X+Y-4Z, -2W+X-Y+2Z, 2W+X-Y-2Z, -2S+X+Y).$$
\end{theorem}

We next define the Lee weight enumerator of a code over $\R$:
\begin{definition}
Let $C$ be a linear code over $\Z_4$. Then the Lee weight enumerator of $C$ is given by
\begin{equation}\label{Leeenum}
Lee_C(W,X) = \sum_{\overline{c} \in C}W^{4n-w_L(\overline{c})}X^{w_L(\overline{c})}.
\end{equation}
\end{definition}
Considering the weights that the variables $X,Y,Z,W,S$ of the symmetrized weight enumerator represent, we easily get the following theorem:
\begin{theorem}\label{swelee}
Let $C$ be a linear code over $\R$ of length $n$.
Then
$$Lee_C(W,X) = swe_C(W^4, X^4, WX^3,W^3X,W^2X^2). $$
\end{theorem}
Now combining Theorem \ref{sweMacWilliams} and Theorem \ref{swelee} we obtain the following theorem:
\begin{theorem}\label{leeMacWilliams}
Let $C$ be a linear code over $\R$ of length $n$ and $C^{\perp}$ be its dual. With $Lee_C(W,X)$ denoting its  Lee weight enumerator as was given in (\ref{Leeenum}), we have
$$Lee_{C^{\perp}}(W,X) = \frac{1}{|C|}Lee_C(W+X,W-X).$$
\end{theorem}
\begin{proof}
Looking at Theorem \ref{sweMacWilliams} and Theorem \ref{swelee}, the proof is complete after observing the following identities:
$$6W^2X^2+4W^3X+W^4+X^4+4WX^3 = (W+X)^4,$$
$$6W^2X^2-4W^3X+W^4+X^4-4WX^3 = (W-X)^4,$$
$$-2W^3X+W^4-X^4+2WX^3 = (W+X)(W-X)^3, $$
$$2W^3X+W^4-X^4-2WX^3 = (W+X)^3(W-X),$$
and finally
$$-2W^2X^2+W^4+X^4 = (W+X)^2(W-X)^2.$$
\end{proof}

\section{Self-dual Codes over $\R$, Projections, lifts and the $\Z_4$-images}

We start by recalling that a linear code $C$ over $\R$ is called {\it self-orthogonal} if $C \subseteq C^{\perp}$ and it will be called {\it self-dual} if $C=C^{\perp}$.

Since the code of length $1$ generated by $u$ is a self-dual code over $\R$, by taking the direct sums, we see that
\begin{theorem}
Self-dual codes over $\R$ of any length exist.
\end{theorem}

We next observe:
\begin{theorem}{\bf (i)} If $C$ is self-orthogonal, then for every codeword $\overline{c} \in C$, $n_{\mathfrak{U}_i}(\overline{c})$ must be even. Here, $n_{\mathfrak{U}_i}(\overline{c})$ denotes the number of units of the $i$th type(in $\mathfrak{U}_i$) that appear in $\overline{c}$
\par {\bf (ii)} If $C$ is self-dual of length $n$, then the all $2u$-vector of length $n$ must be in $C$.
\end{theorem}

\begin{proof}
{\bf (i)} If $C$ is self-orthogonal, then $\langle\overline{c}, \overline{c}\rangle = 0$ for all $\overline{c} \in C$. But by \ref{squareelements} we see that
$$\langle\overline{c}, \overline{c}\rangle = n_{\mathfrak{U}_1}(\overline{c})+n_{\mathfrak{U}_2}(\overline{c})+n_{\mathfrak{U}_2}(\overline{c})\cdot 2u  = 0$$ in $\R$ implies that $n_{\mathfrak{U}_2}(\overline{c})$ must be even and since $n_{\mathfrak{U}_1}(\overline{c})+n_{\mathfrak{U}_2}(\overline{c})$ must also be even, we see that $n_{\mathfrak{U}_1}(\overline{c})$ is also even.
\par{\bf (ii)} If $C$ is self-dual, then again $\langle \overline{c}, \overline{c}\rangle = 0$ for all $\overline{c} \in C$ and so by the above equation, the number of units in $\overline{c}$, which we denote by $n_U(\overline{c})$ must be even. From the ring structure we see that $unit\cdot(2u) = 2u$ and $(non-unit)\cdot (2u) = 0$, thus if we denote the all $2u$-vector of length $n$ by $\overline{2u}$, then
$$\langle\overline{c}, \overline{2u}\rangle = n_U(\overline{c})\cdot (2u) = 0$$ in $\R$ since $n_U(\overline{c})$ is even. This shows that $\overline{2u} \in C^{\perp} = C$ since $C$ is self-dual.
\end{proof}

Define two maps from $(\R)^n$ to $\Z_4^n$ as follows:
\begin{equation}
\mu(\overline{a}+u\overline{b}) = \overline{a}
\end{equation}
and
\begin{equation}
\nu(\overline{a}+u\overline{b}) = \overline{b}.
\end{equation}
Note that $\mu$ is a projection of $\R$ to $\Z_4$. We can define another projection by defining $\alpha:\R \rightarrow \F_2+u\F_2$ by reducing elements of $\Z_4+u\Z_4$ modulo $2$. The map $\alpha$ can be extended linearly like $\mu$. Any linear code over $\R$ has two projections defined in this way.
Since these maps are linear, we see that
\begin{theorem}
If $C$ is a linear code over $\R$ of length $n$, then $\mu(C)$, $\nu(C)$ are both linear codes over $\Z_4$ of length $n$, while $\alpha(C)$ is a linear code over $\F_2+u\F_2$ of length $n$.
\end{theorem}

The following theorem describes the images of self-dual codes over $\R$:

\begin{theorem}\label{main}
Let $C$ be a self-dual code over $\R$ of length $n$. Then
\par {\bf a)} $\phi(C)$ is a formally self-dual code over $\Z_4$ of length $2n$.
\par {\bf b)} $\mu(C)$ is a self-orthogonal code over $\Z_4$ of length $n$ and $\alpha(C)$ is self orthogonal over $\F_2+u\F_2$.
\par {\bf c)} If $\nu(C)$ is self-orthogonal, then $\phi(C)$ is a self-dual code of length $2n$.
\end{theorem}
\begin{proof}{\bf a)} We know that $\phi(C)$ is a linear code over $\Z_4$ of length $2n$. Since $C$ is self-dual,
we know by the MacWilliams identity that was proved in section 3, that the weight enumerator of $C$ is invariant under the MacWilliams transform. The result now follows because $\phi$ is weight-preserving.
\par {\bf b)} Suppose $\overline{a}_1$ and $\overline{a}_2$ are in $\mu(C)$. This means that there exist $\overline{b}_1, \overline{b}_2$ in $\Z_4^n$ such that $\overline{a}_1+u\overline{b}_1, \overline{a}_2+u\overline{b}_2 \in C$. However since $C$ is self-dual we must have
$$\langle \overline{a}_1+u\overline{b}_1,\overline{a}_2+u\overline{b}_2\rangle = 0$$
which means
$$\overline{a}_1 \cdot \overline{a}_2+u (\overline{a}_1 \cdot \overline{b}_2+ \overline{a}_2 \cdot \overline{b}_1) = 0$$ from which
$$\overline{a}_1 \cdot \overline{a}_2  =0 $$ follows. Here, $\cdot$ stands for the euclidean dot product in $\Z_4.$
A similar proof can be done for $\alpha$ as well.
\par {\bf c)} Now since $C$ is self-dual, for all $\overline{a}_1+u\overline{b}_1, \overline{a}_2+u\overline{b}_2 \in C$, we have
$$\langle \overline{a}_1+u\overline{b}_1,\overline{a}_2+u\overline{b}_2 \rangle = \overline{a}_1 \cdot \overline{a}_2+u (\overline{a}_1 \cdot \overline{b}_2+
\overline{a}_2 \cdot \overline{b}_1) = 0,$$
from which it follows that
\begin{equation}
\overline{a}_1 \cdot \overline{a}_2  = \overline{a}_1 \cdot \overline{b}_2+ \overline{a}_2 \cdot \overline{b}_1 = 0.
\end{equation}
But by the hypothesis, we also have $\overline{b}_1 \cdot \overline{b}_2 = 0$ for any such codewords. Combining these all, we get
\begin{align*} \phi(\overline{a}_1+u\overline{b}_1)\cdot \phi(\overline{a}_2+u\overline{b}_2) & = (\overline{b}_1,\overline{a}_1+\overline{b}_1)\cdot
(\overline{b}_2,\overline{a}_2+\overline{b}_2)\\
& = 2\overline{b}_1 \cdot \overline{b}_2 + \overline{a}_1 \cdot \overline{a}_2+\overline{a}_1 \cdot \overline{b}_2+ \overline{a}_2 \cdot \overline{b}_1 \\
& = 0.
\end{align*}
This proves that $\phi(C)$ is self-orthogonal. But since $\phi$ is an isometry, both $C$ and $\phi(C)$ have the same size, so
$|C| = |\phi(C)| = 16^{n/2} = 4^n$, which proves that $\phi(C)$ is self-dual.
\end{proof}

\begin{corollary}
If $C$ is a self-dual code over $\R$, generated by a matrix of the form $[I_n|A]$, then $\mu(C)$ and $\alpha(C)$ are self-dual over $\Z_4$ and $\F_2+u\F_2$ respectively.
\end{corollary}

If $\mu(C) = D$ and $\alpha(C) = E$, we say that $C$ is a {\it lift} of $D$ and $E$. One way of obtaining good codes over $\R$ is to take the good ones over $\Z_4$
and $\F_2+u\F_2$ and take their lift over $\R$. The following theorem gives us a bound on how good the lift can be:
\begin{theorem}
Let $D$ be a non-zero linear code over $\Z_4$ and $E$ be a non-zero linear code over $\F_2+u\F_2$ such that $\mu(C) = D$ and $\alpha(C) = E$. Let $d,d',d''$ denote the minimum Lee weights of $C$,  $D$ and $E$ respectively. Then $d \leq 2d'$ and $d\leq 2d''$.
\end{theorem}

\begin{proof}
Let $\overline{x} \in D$ be such that $w_L(\overline{x}) = d'$. Since $\mu(C) = D$, $\exists \overline{y}\in \Z_4^n$ such that $\overline{x}+u\overline{y} \in D$. Since $C$ is linear, we have $u(\overline{x}+u\overline{y}) = u \overline{x} \in C$. But by definition of the Lee weight on $\R$, we have
$w_L(u \overline{x}) = w_L(\overline{x}, \overline{x}) = 2d'$.
Thus the inequality $d\leq 2d'$ is proved.

The second inequality is proved in exactly the same way.
\end{proof}
As the theorem suggests, to construct good codes over $\R$ by lifting from $\Z_4$ and $\F_2+u\F_2$, we need to take codes over the projection rings that have high minimum distances.
We illustrate this in the following example:
\par {\bf Example} Let $D$ be the $\Z_4$-code generated by $G' = [I_8|A']$  and $E$ be the $\F_2+u\F_2$-linear code generated by $G''=[I_8|A'']$ where
$$A' = \left [
\begin{array}{cccccccc}
0&2&3&0&0&1&3&2 \\
2&0&2&3&0&0&1&3\\
3&2&0&2&3&0&0&1\\
1&3&2&0&2&3&0&0\\
0&1&3&2&0&2&3&0\\
0&0&1&3&2&0&2&3\\
3&0&0&1&3&2&0&2\\
2&3&0&0&1&3&2&0
\end{array}
\right], \:\:\:A'' = \left [
\begin{array}{cccccccc}
u&u&1&u&0&1&1&u \\
u&u&u&1&u&0&1&1\\
1&u&u&u&1&u&0&1\\
1&1&u&u&u&1&u&0\\
0&1&1&u&u&u&1&u\\
u&0&1&1&u&u&u&1\\
1&u&0&1&1&u&u&u\\
u&1&u&0&1&1&u&u
\end{array}
\right].$$
$D$ and $E$ are both codes of length $16$, size $4^8$ and minimum Lee distance $8$. We consider a common lift of $D$ and $E$ over $\Z_4+u\Z_4$ to obtain $C$ that is generated by
$G = [I_8|A]$, where
$$A = \left [
\begin{array}{cccccccc}
u&2+u&3+2u&u&0&1&3&2+u \\
2+u&u&2+u&3+2u&u&0&1&3\\
3&2+u&u&2+u&3+2u&u&0&1\\
1&3&2+u&u&2+u&3+2u&u&0\\
0&1&3&2+u&u&2+u&3+2u&u\\
u&0&1&3&2+u&u&2+u&3+2u\\
3+2u&u&0&1&3&2+u&u&2+u\\
2+u&3+2u&u&0&1&3&2+u&u
\end{array}
\right].$$
$C$ is a linear code over $\Z_4+u\Z_4$ of length $16$, size $(16)^8 = 4^{16}$ and minimum Lee distance $12$. Taking the Gray image, we get $\phi(C)$ to be a formally self-dual $\Z_4$-code of length $32$ and minimum Lee distance $12$.

\section{Three constructions for formally self-dual codes over $\R$}
Because of Theorem 3.10 we can easily show that the Gray image of formally self-dual codes over $\R$ are formally self-dual over $\Z_4$ as well. Some of the construction methods described in \cite{Huffman-Pless} for binary codes can be extended to $\Z_4+u\Z_4$ as well.
\begin{theorem}
Let $A$ be an $n\times n$ matrix over $\R$ such that $A^{T} = A$. Then the code generated by the matrix $[I_n\ | \ A]$ is a formally self-dual code of length $2n$.
\end{theorem}
\begin{proof}
Consider the matrix $[-A^T\ | \ I_n] = [-A\ | \ I_n] = G'$ and let $G = [I_n\ | \ A]$. Both $G$ and $G'$ generate codes of size $16^n$. Moreover the codes generated are equivalent over $\Z_4+u\Z_4$. So if $C = \langle G \rangle$ and $C' = \langle G'\rangle$, all we need to do to finish the proof is to show that $C' = C^{\perp}$.

Let $\overline{v}$ be the $i$-th row of $G$ and $\overline{w}$ be the $j$-th row of $G'$. Then $\langle\overline{v}, \overline{w}\rangle_k = -A_{ji}+A_{ij} = 0$ since $A^T = A$. Therefore $C'  = C^{\perp}$ and $C$ is equivalent to $C' = C^{\perp}$. Since $w_L(-a)= w_L(a)$ for all $a \in \R$, this is a weight preserving equivalence.
\end{proof}

\begin{theorem} Let $M$ be a circulant matrix over $\R$ of order $n$. Then the matrix $[I_n\ | \ M]$ generates a formally self-dual code over $\R$. This is called the double circulant construction.
\end{theorem}
\begin{proof}
Let $C$ be the code generated by $[I_n\ | \ M]=G$ and let $C'$ be the code generated by $[-M^{T}\ | \ I_n] =G'$. It can similarly be shown that $C' = C^{\perp}$. Observe that $C'$ is equivalent over $\Z_4+u\Z_4$ to the code $C''$ generated by $[M^T|I_n] = G''$. So to show the equivalence of $C$ and $C'$, we will show the equivalence of $C$ and $C''$. There is a permutation $\sigma$ of rows such that after applying it to $G''$, the first column of $\sigma(M^{T})$ is the same as the first column of $M$. Namely,
$$(M_{11},M_{21}, \cdots, M_{n1}) = (M^T_{\sigma(1)1}, M^T_{\sigma(2)1}, \cdots, M^T_{\sigma(n)1}) = (M_{1\sigma(1)}, M_{1\sigma(2)}, \cdots, M_{1\sigma(n)}).$$
Since the matrix $M$ is circulant, every column of $M$ is then equal to a column of $\sigma(M^T)$. Then apply the necessary column permutation $\tau$ so that $\tau(\sigma(M^T)) = M$. We apply another column permutation $\rho$ so that the matrix found by $\sigma(I_n)$ is returned to identity. Note that $\tau$ does not affect this part. Thus we obtain $G$ from $G''$ by the consecutive applications of the permutations $\sigma, \tau$ and $\rho$, which means that $C$ is equivalent to $C'= C^{\perp}$, and again as in the first theorem this is a weight preserving equivalence.
\end{proof}

Now, applying Theorem \ref{GrayImage}, we have the following corollary:
\begin{corollary}
Let $C$ be a linear code over $\R$ generated by a matrix of the form $[I_n|A]$, where $A$ is an $n\times n$ matrix. If $A$ is symmetric or circulant, then C is formally self-dual and hence $\phi(C)$ is a formally self-dual code over $\Z_4$ of length $4n$.
\end{corollary}

\begin{theorem}\label{bordered}
Let $M$ be a circulant matrix over $\R$ of order $n-1$. Then the matrix
$$G= \left [
\begin{array}{cccc|ccccc}
&  &  &  & \alpha & \beta & \beta & ...  &\beta  \\
&  &  &  & \gamma &  &  &  &  \\
&  & I_n &  & \gamma &  & M &  &  \\
&  &  &  & . &  &  &  &  \\
&  &  &  & . &  &  &  &  \\
&  &  &  & \gamma &  &  &  &  \\
\end{array} \right],
$$ where $\alpha, \beta, \gamma \in \R$ such that $\gamma = \pm \beta$, generates a formally self-dual code of length $2n$ over $\Z_4+u\Z_4$ whose Gray image is a formally self-dual code over $\Z_4$ of length $4n$. This is called the bordered double circulant construction.
\end{theorem}
\begin{proof}
Let
$$G' = \left [
\begin{array}{ccccc|cccc}
-\alpha & -\gamma & -\gamma & ...  &-\gamma &  &  &  &  \\
-\beta &  &  &  &  &  &  &  & \\
 -\beta &  & -M^{T} &  &  &  & I_n &  &\\
 . &  &  &  & &  &  &  & \\
 . &  &  &  &  &  &  &  &\\
 -\beta &  &  &  & &  &  &  & \\
\end{array} \right].
$$ It is easy to see that $G'$ generates $C^{\perp}$.
By the same method as was done in the previous theorem the parts of $G$ and $G'$ except the $\beta$ and $\gamma$ can be made equuivalent. Multiplying all the columns except the $I_n$ by $-1$ we see that $C^{\perp}$ is equivalent to a code $C^*$ with generator matrix of the form
$$G^* = \left [
\begin{array}{cccc|ccccc}
&  &  &  & \alpha & \gamma & \gamma & ...  &\gamma  \\
&  &  &  & \beta &  &  &  &  \\
&  & I_n &  & \beta &  & M &  &  \\
&  &  &  & . &  &  &  &  \\
&  &  &  & . &  &  &  &  \\
&  &  &  & \beta &  &  &  &  \\
\end{array} \right].$$
If $\gamma = \beta$, it is easy to see that $C$ and $C^*$ will be the same codes. If $\gamma = -\beta$, then multiply all but the first row of $G^*$ by $-1$. The resulting matrix still generates $C^*$. But since $w_L(a) = w_L(-a)$ for all $a \in \R$, we see that $C$ and $C^*$ will have the same weight enumerator. Hence in both cases we see that $C$ and $C^{\perp}$ have the same weight enumerators.
\end{proof}

We finish with some examples of formally self-dual $\Z_4$-codes thus obtained:

\bigskip
\noindent
{\bf Formally Self-dual codes from double circulant constructions}

We present in the following, some good formally self-dual codes over $\Z_4$ obtained from formally self-dual codes constructed over $\R$ by the purely double circulant matrices. So the code $C$ is generated over $\R$ by a matrix of the form $[I_n|M]$ where $M$ is a circulant matrix, hence in the tables we will only give the first row of $M$. The length indicates the length of the code over $\R$, hence the length of the $\Z_4$-image is doubled. We will also indicate the minimum Lee distance of the codes:
\bigskip
\noindent
\begin{table}[H]
 \caption{Good f.s.d $\Z_4$-codes obtained from double circulant matrices over $\R$}
\begin{tabular}
[c]{|c|c|c|}\hline
Length & First Row of $M$ & $d$ \\\hline
$4$ & ($2$,$1+2u$)& $4$ \\\hline
$6$ & ($2$,$1$,$3u$)& $6$ \\\hline
$8$ & ($3+3u$,$3u$,$2u$,$2+3u$)& $8$ \\\hline
$10$ & ($1$,$0$,$2$,$3u$,$2+u$)& $8$ \\\hline
$12$ & ($0$,$2$,$3$,$2u$,$3$,$u$)& $10$ \\\hline
$14$ & ($3+3u$,$3+3u$,$1+2u$,$1$,$2+2u$,$3$,$3$)& $11$ \\\hline
$16$ & ($0$,$0$,$1+2u$,$1+2u$,$1$,$1$,$3u$,$1+u$)& $12$ \\\hline
$18$ & ($0$,$0$,$1$,$1$,$1+2u$,$3+3u$,$2+2u$,$1+u$,$2$)& $12$ \\\hline
$20$ & ($0$,$0$,$1$,$3$,$1$,$3+2u$,$u$,$3+2u$,$u$,$2+u$)& $14$ \\\hline
$22$ & ($0$,$0$,$1$,$1$,$1$,$1$,$2$,$1$,$2+2u$,$1+3u$,$3+2u$)& $14$ \\\hline
$24$ & ($0$,$0$,$1$,$1$,$1$,$1$,$0$,$1$,$0$,$2$,$2u$,$2+3u$)& $14$ \\\hline
$26$ & ($0$,$0$,$1$,$1$,$1$,$1$,$0$,$3$,$1+u$,$2u$,$3u$,$1+2u$,$3+2u$)& $15$ \\\hline
\end{tabular}
\end{table}

\bigskip
\noindent
{\bf Formally Self-dual codes from bordered double circulant constructions}

We present in the following, some good formally self-dual codes over $\Z_4$ obtained from formally self-dual codes constructed over $\R$ by bordered double circulant matrices as given in Theorem \ref{bordered}. So the code $C$ is generated over $\R$ by a matrix of the form given in Theorem \ref{bordered}.Hence in the tables we will give $\alpha, \beta, \gamma$ values as well as the first row of $M$. Again the length of the $\Z_4$-images is double the length indicated. $d$ denotes the minimum Lee distance:

\bigskip
\noindent
\begin{table}[H]
 \caption{Good f.s.d $\Z_4$-codes obtained from bordered double circulant matrices over $\R$}
\begin{tabular}
[c]{|c|c|c|c|}\hline
Length & First Row of $M$ & $(\alpha, \beta,\gamma)$& $d$ \\\hline
$4$ & ($0$)& $(0, 1+2u,1+2u)$ & $4$\\\hline
$6$ & ($2u$,$1$)& $(3+3u,1+3u,1+3u)$ & $6$\\\hline
$8$ & ($3+3u$,$3+2u$,$u$)& $(2,3+2u,3+2u)$ & $8$\\\hline
$10$ & ($0$,$0$,$1+2u$,$1$)& $(3,1+2u,1+2u)$ & $8$\\\hline
$12$ & ($1+2u$,$1$,$2$,$1+3u$,$3$)& $(u,1+2u,1+2u)$ & $10$\\\hline
$14$ & ($0$,$0$,$u$,$u$,$2$,$3+2u$)& $(3+u,1+2u,1+2u)$ & $10$\\\hline
$16$ & ($1$,$1$,$0$,$1$,$3+u$,$3u$,$1+2u$)& $(3+2u,1,1)$ & $11$\\\hline
$18$ & ($0$,$0$,$0$,$0$,$2+2u$,$3u$,$1$,$3+2u$)& $(3+u,3+2u,3+2u)$ & $12$\\\hline
$20$ & ($0$,$0$,$0$,$0$,$u$,$1+3u$,$1+u$,$u$,$2+2u$)& $(1+u,3+2u,3+2u)$ & $12$\\\hline
$22$ & ($0$,$0$,$0$,$0$,$2u$,$1+u$,$3+u$,$1+3u$,$2+2u$,$2+3u$)& $(1+u,3+2u,3+2u)$ & $14$\\\hline
$24$ & ($0$,$0$,$0$,$0$,$0$,$1$,$u$,$2u$,$2+2u$,$2+3u$,$3$)& $(1,1+2u,1+2u)$ & $14$\\\hline
\end{tabular}
\end{table}

\begin{remark}
The first construction, namely the construction by symmetric matrices did not lead to good numerical results. That is why we did not tabulate codes from the first construction.
\end{remark}
\begin{remark}
In both the tables given above, an exhaustive search was made through all possible codes of the given form up to length $14$, but for higher lengths the search was not exhaustive. For these higher lengths, an exhaustive search can be made using more powerful computational resources with a possibility of better minimum distances.
\end{remark}

\bibliographystyle{amsplain}

\begin{thebibliography}{15}

\bibitem{us}  S.T.Dougherty, B.Yildiz and S.Karadeniz, \emph{Codes over $R_k$, Gray Maps and their Binary Images}, Finite Fields Appl., \textbf{17}, 205--219 (2011).

 \bibitem{Dougherty}
S.T.~Dougherty, P.~Gaborit, M.~Harada and P.~Sol\'{e}, \emph{Type
II codes over $\mathbb{F}_2+u\mathbb{F}_2$}, IEEE Trans. Inform.
Theory, \textbf{45}, 32--45 (1999).

\bibitem{Duursma} I.M. ~Duursma, M. ~Greferath and S. E. Schmidt, \emph{On the Optimal $\Z_4$-codes of TypeII and length 16}, J. Combin. Theory, Series A, \textbf{92}, 77--82 (2000).

\bibitem{Gulliver} T.A. ~Gulliver and M. ~Harada, \emph{Optimal Double Circulant $\Z_4$-codes}, LNCS:AAAAECC, \textbf{2227}, 122-128 (2001).

\bibitem{Sloane:Gray}
A.R. ~Hammons, V.~Kumar, A.R.~Calderbank, N.J.A.~Sloane, and P.
~Sol\'{e}, \emph{The $\Z_4$-linearity of Kerdock, Preparata,
Goethals and related codes}, IEEE Trans. Inform. Theory,
\textbf{40}, 301--319 (1994).

\bibitem{Huffman-Pless} W.C. Huffman and V. Pless,
{\sl Fundamentals of Error Correcting Codes.}
Cambridge: University Press, 2003.

\bibitem{Huffman}
W.C.~Huffman, \emph{Decompositions and extremal Type II codes over
$\Z_4$}, IEEE Trans. Inform. Theory, \textbf{44}, 800--809 (1998).

\bibitem{68} S. Karadeniz and B. Yildiz, \emph{New extremal binary self-dual codes of length 68 from R2-lifts of binary self-dual codes}, Adv. Math. Commun., \textbf{7}, 219--229 (2013).

\bibitem{bordered} S.Karadeniz and B.Yildiz, \emph{Double-Circulant and Double-Bordered-Circulant constructions for self-dual codes over $R_2$}, Adv. Math. Commun., \textbf{6}, 193--202 (2012).

\bibitem{Wan}
Z.X.~Wan, \emph{Series on Applied Mathematics:Quaternary Codes}, World Scientific, 1997.

\bibitem{Wolf} J.Wolfmann, \emph{Negacyclic and Cyclic Codes over $\Z_4$},
IEEE Trans. Inform. Theory, \textbf{45}, 2527--2532 (1999).

\bibitem{Wood} J. Wood,
\emph{ Duality for modules over finite rings and applications to
coding theory. } Amer.  J.  Math., \textbf{121}, 555-575 (1999).

\bibitem{YKilk} B. Yildiz and S. Karadeniz, \emph{Linear codes over $\F_2+u\F_2+v\F_2+uv\F_2$}, Des. Codes Crypt.,
\textbf{54}, 61--81 (2010).

\end{thebibliography}

\end{document}